\documentclass{article}

\usepackage[utf8]{inputenc}

\usepackage{amsmath}
\usepackage{amsfonts}
\usepackage{amssymb}

\usepackage{amsmath, amsthm, amssymb, amsfonts}
\usepackage{graphicx, color} 
\usepackage{fancyhdr} 
\usepackage{caption, subcaption} 
\usepackage{enumerate} 
\usepackage{mathrsfs} 
\usepackage{makeidx} 
\usepackage[section]{placeins} 
\usepackage[all]{xy} 
\usepackage{appendix}

\usepackage{verbatim}
\usepackage{moresize}

\usepackage{hyperref}
\hypersetup{pdfpagemode={UseOutlines},
bookmarksopen=true,
bookmarksopenlevel=0,
hypertexnames=false,
colorlinks=true,
citecolor=magenta,
linkcolor=red,
urlcolor=magenta,
pdfstartview={FitV},
unicode,
breaklinks=true,
}

\newtheorem{prop}{Proposition}

\newtheorem{theorem}[prop]{Theorem}
\newtheorem{coro}[prop]{Corollary}
\newtheorem{lema}[prop]{Lemma}
\theoremstyle{definition}
\newtheorem{de}[prop]{Definition}

\newtheorem{obs}[prop]{Remark}

\newtheorem{pregunta}[prop]{\it Question}

\newcommand{\ital}{\mathcal}

\newcommand{\cantor}{2 ^\omega}

\newcommand{\tits}[1]{{\itshape #1}}

\newcommand{\rtt}{\upharpoonright} 

\newcommand{\gorro}{\widehat}

\newcommand{\fin}{[\omega]^{< \omega}}
\newcommand{\dosfin}{2^{< \omega}}

\newcommand{\ada}{\mathcal{A}}
\newcommand{\adb}{\mathcal{B}}
\newcommand{\rest}{\upharpoonright}

\usepackage{cite}

\title{$\mathfrak{c}$-many types of a $\Psi$-space}
\author{Héctor Alonzo Barriga-Acosta, Fernando Hernández-Hernández \\ Posgrado Conjunto en Ciencias Matemáticas UNAM-UMSNH, México}
\begin{document}

\maketitle

\begin{abstract}
We show that for any cardinal $\omega<\kappa \leq \mathfrak{c}$ with $cf(\kappa) > \omega$, there are $\mathfrak{c}$ many AD families whose $\Psi$-spaces are pairwise non-homeomorphic and they can be Luzin families or branch families of $2^\omega$.
\end{abstract}

\section{Introduction}
Denote by $\fin$ the set of finite sets of the natural numbers,  $\omega$. If $\ada$ is an almost disjoint family (AD family, for short) on $\omega$, define the topological space $\Psi (\ada) = (\omega \cup \ada,  \tau)$ as follows: $\omega$ is a discrete subset of $\Psi (\ada)$; basic neighborhoods of a point $x\in \ada$ are of the form $\{ x\} \cup (x\setminus F)$, where $F \in \fin$. 

$\Psi$-spaces have been well studied through the years because they are candidates to give examples or counterexamples of many topological concepts. There are nice properties $\Psi$-spaces satisfy: they are Hausdorff, separable, first countable, locally compact, zero dimensional. For topological and combinatorial aspects of $\Psi$-spaces see \cite{FM} and \cite{M}, respectively. 

Daniel Bernal-Santos and Salvador García-Ferreira wondered if $C_p (\Psi (\ada))$ and $C_p (\Psi (\adb))$ are homeomorphic whenever $\ada$ and $\adb$ are homeomorphic as subspaces of $\cantor$. To understand better the space $C_p (\Psi (\ada))$ they wondered for a more elementary question: 

\begin{pregunta}[Bernal-Santos, García-Ferreira] \label{Q1}
If $X, Y \subseteq \cantor$ are homeomorphic, are $\Psi (\ada_X)$ and $\Psi (\ada_Y)$ homeomorphic?
\end{pregunta}  

Here, $\ada_X$ is the family of branches determined by $X$ as a subspace of $\cantor$.   It is well known that under $\mathsf{MA +\neg CH}$, every set $X\in \cantor$ of size less than the continuum is a $Q$-set, and thus, $\Psi (\ada_X)$ normal (call a family $\ada$ {\it normal} if its $\Psi$-space is normal). Like this, there are many topological properties of $X\subseteq\cantor$ that have effect on the $Psi$-space $\Psi (\ada_X)$. One might think that $\mathsf{MA +\neg CH}$ is a good ingredient to conjecture that the answer is affirmative. However, we answer negatively to Question \ref{Q1} since Theorem \ref{T2} shows that in $\mathsf{ZFC}$ there are different types of spaces $\Psi (\ada_X), \Psi (\ada_Y)$ even when $X$ and $Y$ are homeomorphic.

Recall that an AD family $\ada$ is \tits{Luzin} if it can be enumerated as 
$$
\ada = \langle A_\alpha : \alpha < \omega_1 \rangle
$$ 
in such way that $\forall \alpha < \omega_1 \forall n\in \omega \; \; (|\{ \beta < \alpha : A_\alpha \cap A_\beta \subseteq n  \}|< \omega)$. Normal and Luzin families are in some sense ``orthogonal", precisely because the normality of their $\Psi$-spaces holds in the former and breaks down badly in the last. We show in Theorem \ref{T1} that in $\mathsf{ZFC}$ there are different types of $\Psi$-spaces for Luzin families.

Focusing on AD families of size $\omega_1$, Michael Hru\v{s}\'ak formulated the following question:

\begin{pregunta}[Hru\v{s}\'ak] \label{Q2}
Is consistent that there is an almost disjoint family $\ada$  such that $\Psi (\ada) \simeq \Psi (\adb)$, whenever $\adb\subseteq\ada$ and $\vert\ada\vert=\vert\adb\vert$\,?
\end{pregunta}

Observe that $2^\omega < 2^{\omega_1}$ (in particular $\mathsf{CH}$), implies that the answer to Question \ref{Q2} is negative by the simple fact that given $\ada$ of size $\omega_1$, there are only $\mathfrak{c}$ many subspaces $\Psi(\adb)$ for which $\Psi(\ada) \simeq \Psi(\adb)$ (there are only $\mathfrak{c}$ permutations of $\omega$), and there are $2^{\omega_1}$ many subsets of $\ada$ of size $\omega_1$. We believe that it is a very interesting question; we conjecture that the answer is no but our methods do not work to solve it.

\section{Basic facts}

Our notation is standard and follows closely \cite{FM} and \cite{M}. We use, like them, $f(A)$ to denote the evaluation of the function $f$ at the point $A$ in its domain while $f[A]$ denotes the image of the set $A$ under the function $f$. For sets $A$ and $B$, we say that $A\subseteq^* B$, in words that  $A$ is almost contained in $B$, if $A\setminus B$ is a finite set.  Likewise, $A=^*B$ if and only if $A\subseteq^* B$ and $B\subseteq^* A$. If $X \subseteq \cantor$, $x\in \cantor$, we denote 
\[
\gorro{x\downarrow n} = \{ x\rest k \in \dosfin: n\leq k \}, \;
\]
$\gorro{x} :=  \gorro{x\downarrow 0}, \; \gorro{X} = \bigcup_{x\in X} \gorro{x}$ and $\mathcal{A}_X = \{ \gorro{x} : x\in X \}$. Families of the form $\ada_X$, where $X\subseteq \cantor$, are canonical AD families on $\dosfin$, and there are of any size below the continuum. Under a correspondence between $\omega$ and $\dosfin$ we can consider $\Psi(\ada_X)$. Perhaps the families $\ada_X$ were first studied by F. Tall \cite{T} when he showed that if $X\subseteq  \cantor$, then $X$ is a $Q$-set if and only if $\Psi (\ada_X)$ is normal.

   The following Lemma shows how a homeomorphism between $\Psi$-spaces looks like.

\begin{lema}\label{L0}
Let $\ital{A}, \ital{B}$ almost disjoint families on $\omega$, and $H: \Psi (\mathcal{A}) \to \Psi (\mathcal{B})$ be bijective. Then $H$ is a homeomorphism if and only if $H[\omega] = \omega$ and for every $x\in \mathcal{A}$, $H[x]$ and $H(x)$ as subsets $\omega$ are almost equal. 
\end{lema}

\begin{proof}
$\Rightarrow )$ Since $H$ must send isolated points to isolated points, it is clear that $H[\omega] = \omega$. Now, let $x\in \mathcal{A}$ and $\{ H(x)\} \cup (H(x) \setminus F)$ be a neighborhood of $H(x)$, where $F\in \fin$. By continuity, there is $F'\in \fin$ such that $H(x) \cup H[x\setminus F'] = H[\{ x\} \cup (x\setminus F')]  \subseteq \{ H(x)\} \cup (H(x) \setminus F)$. So, $H[x] \setminus F'=H[x\setminus F'] \subseteq H(x) \setminus F$. Use the fact that $H$ is open and similar arguments to get $H[x] \supseteq^* H(x)$. 
\\ \\
$\Leftarrow )$ We will see that $H$ is continuous; to see that $H$ is open use similar arguments. Let $x\in \mathcal{A}$ and $\{ H(x) \} \cup (H(x) \setminus F )$ be a neighborhood of $H(x)$, where $F\in \fin$. Since $H[x] =^* H(x)$, there is $F' \in \fin$ such that $H[x\setminus F'] =H[x]\setminus F' \subseteq H(x)\setminus F$. So, $H[\{ x\} \cup (x\setminus F')] \subseteq \{ H(x) \} \cup (H(x) \setminus F)$.
\end{proof}

For $s\in \dosfin$, let $\langle s \rangle = \{ t\in \dosfin : s \subseteq t \}$ and $[\langle s \rangle ] = \{ x\in \cantor : s \subseteq x \}$.

\begin{lema}
Let $X\subseteq \cantor$ be a set of size $\kappa$, $cf(\kappa) > \omega$. Then there are infinitely many $n\in \omega$ for which there are $s, t\in 2^n$ such that $|[\langle s \rangle] \cap X| = \kappa = |[\langle t \rangle ] \cap X|$.
\end{lema}

\begin{proof}
Suppose for a contradiction that for every $ n\in \omega$ there is a unique $s_n\in 2^n$ such that $X_n =[\langle s_n \rangle] \cap X$ has size $\kappa$. Let $Y_n = X \setminus X_n$. Notice that $|Y_n|< \kappa$, and since $cf(\kappa) > \omega$, $Y = \bigcup_{n\in \omega} Y_n$ has size less than $\kappa$. This is a contradiction because $ X \setminus Y = \bigcap_{n\in \omega} X_n \subseteq  \bigcap_{n\in \omega} [\langle s_n \rangle]$ has size $\kappa$ and it contains at most one point.  
\end{proof}

In the following, we are going to consider AD families of a fixed size $\kappa \leq \mathfrak{c}$, with $cf(\kappa) > \omega$.

Notice that by the previous Lemma, one can actually get infinitely many $n\in \omega$ for which there is $s\in 2^n$ such that $|[\langle s^\frown 0 \rangle] \cap X| = \kappa = |[\langle s^\frown 1 \rangle ] \cap X|$. For an AD family $\ada$ on $\omega$, we obtain the next observation by considering $\{ \chi (A) : A \in \ada \} \subseteq \cantor$, where $\chi$ is the characteristic function.

\begin{obs}\label{L1}
Let $\mathcal{A}$ be an $AD$ family of size $\kappa$ with $cf(\kappa)>\omega$. Then $$\forall n\in \omega\ \exists m > n (|\{ x \in \ada : m \in x  \}|=|\{ x \in \ada : m \notin x  \}| = \kappa).$$
\end{obs}

\begin{lema}\label{L2}
Let $\ada, \adb$ be $AD$ families of size $\kappa$ with $cf(\kappa)>\omega$ and $h : \ada \to \adb$ be a bijection. Then for all $n\in \omega$ there are $x,y,z \in \ada$ such that 
\begin{enumerate}
\item $max\{x\cap y\} > n \ \wedge\ x\cap y \subsetneq x\cap z;$ and
\item $max\{h(x)\cap h(y)\} > n \ \wedge\  h(x)\cap h(y) \subsetneq h(x)\cap h(z).$
\end{enumerate}   
\end{lema}

\begin{proof}
Fix $n\in \omega$. By Remark \ref{L1}, choose $m>n$ large enough and $\ada' \in [\ada]^{\kappa}$ such that for every $x\in \ada'$, $m\in x$ and $m \in h(x)$. Now, fix $y\in \ada'$ and consider $\{ x\cap y : x \in \ada' \wedge x\neq y\}$. There are $F\in \fin$ and $\ada'' \in [\ada']^{\kappa}$ such that for all $x\in \ada''$, $x\cap y = F$. There are also $G \in \fin$ and $\adb ' \in [h[\ada'']]^{\kappa}$ such that for all $w\in \adb'$, $w\cap h(y) = G$. Let $\ada''' = h^{-1}[\adb']$.

For any $x,z \in \ada'''$, we have $x\cap y = z \cap y$ and $h(x)\cap h(y) = h(z) \cap h(y)$. This implies that $x\cap y \subseteq x\cap z$ and $h(x)\cap h(y) \subseteq h(x) \cap h(z)$. Observe that it is possible to find $x,z$ such that the contentions are proper.
\end{proof}


\begin{de}
Let $\ada, \adb$ be AD families on $\omega$ of size $\kappa$, $h: \ada \to \adb$ bijective. We say that $h$ is of \tits{dense oscillation} if for each $\ada' \in [\ada]^\kappa$ there are $ x, y, z\in \ada'$ such that $|x\cap z \setminus x\cap y| \neq |h(x)\cap h(z)\setminus h(x)\cap h(y)|$.
\end{de}

\begin{prop}\label{P1}
Let $\ada, \adb$ be AD families of size $\kappa$ with $cf(\kappa)>\omega$ and $h: \ada \to \adb$ be of dense oscillation. Then, there is no homeomorphism from $\Psi (\ada)$ to $\Psi (\adb)$ that extends $h$.
\end{prop}

\begin{proof}
Suppose for a contradiction that $H: \Psi (\ada) \to \Psi (\adb)$ is a homeomorphism extending $h$. By Lemma \ref{L0}, for every $A\in \ada$, $H[A]=^* H(A)$. So, for $A\in \ada$, consider the finite sets $F_A = \{ n \in A : H(n) \notin H(A) \}$ and $G_A = \{ n\in H(A) : H^{-1}(n) \notin A \}$.

There are $\ada' \in [\ada]^{\kappa}$ and $F,G \in \fin$ such that for all $A\in \ada'$, $F= F_A$ and $G=G_A$.  If $x, y, z \in \ada'$, then $(x\cap z \setminus x\cap y) \cap F = \emptyset $ and 
\[
\Big(\left( H(x)\cap H(z) \right)\setminus \left(H(x)\cap H(y)\right)\Big) \cap G = \emptyset. 
\]
Moreover, $m \in x \setminus F$ implies $H(m) \in H(x)$ and $H(m) \in H(x) \setminus G$ implies $H^{-1}(m) \in x$.  From this, one can deduce that $$|x\cap z \setminus x\cap y| = |H(x)\cap H(z)\setminus H(x)\cap H(y)|,$$ contradicting the dense oscillation property of $H\rest \ada = h$.
\end{proof}


\begin{de}
Let $A, B \subseteq \omega$.
\begin{itemize}
\item $A$ and $B$ are \tits{oscillating} if  
\[
\forall \{ x, y \} \subseteq A\ \forall \{ w, z \} \subseteq B \ (|y-x| \neq |z-w|).
\]

\item $A$ and $B$ are \tits{almost oscillating} if there is $n \in \omega$ such that $  A\setminus n$ and $ B \setminus n$ are oscillating.
\end{itemize}
\end{de}

\begin{prop}\label{P2}
There are $\mathfrak{c}$ many infinite subsets of $\omega$ pairwise almost oscillating.
\end{prop}

\begin{proof}
From $\omega$, we construct first two oscillating sets $A=\bigcup A_n, B= \bigcup B_n$. Fix $A_0 = \{ 0\}, B_0 =\{ 1\}$. Suppose constructed $A_n=\{ a_0, \ldots, a_n\}, B_n=\{ b_0, \ldots, b_n\}$ oscillating. Let $a_{n+1}\in \omega$ such that $a_{n+1} - a_n > b_n - b_0$ and $b_{n+1}\in \omega$ such that $b_{n+1} - b_n > a_{n+1} - a_0$. Observe that $A_{n+1} = A_n \cup \{ a_{n+1} \}, B_{n+1} = B_n \cup \{ b_{n+1} \}$ are oscillating as well as will be $A$ and $B$.

Notice that the construction is hereditary. That is, for any $X\in [\omega]^\omega$, there are $A,B \in [X]^\omega$ oscillating. This allows to define a Cantor tree induced by this partitions. Each branch of the Cantor set, $f\in2^\omega$, represents a decreasing sequence of infinite sets of naturals $\langle A_{f\rtt n} : n\in \omega \rangle$ such that for any other branch $g\in2^\omega$, we have that $A_{f\rtt k}, A_{g\rtt l}$ are oscillating whenever  $k,l > \Delta (f,g)$. Now, for every branch $\langle A_{f\rtt n} : n\in \omega \rangle$ consider a pseudointersection $P_f$ of $\{ A_{f\rtt n} : n\in \omega \}$. Observe that for any two branches $\langle A_{f\rtt n} : n\in \omega \rangle$, $\langle A_{g\rtt n} : n\in \omega \rangle$, their pseudointersections $P_{f}$, $P_{g}$ are almost oscillating. 
\end{proof}

\begin{coro}\label{C1}
Let $\ada, \adb$ be AD families and $h: \ada \to \adb$ a bijection. If $A=\{ |x\cap y| : x,y \in \ada \}$ and $B=\{ |x\cap y| : x,y \in \adb \}$ are almost oscillating, there is $\ada' \in [\ada]^{\kappa}$ such that $h \rest \ada'$ is of dense oscillation.
\end{coro}

\begin{proof}
Let $n \in \omega$ such that $A\setminus n$ and $B \setminus n$ are oscillating. By iterating finitely many steps of Remark \ref{L1}, we can get $\ada' \in [\ada]^{\kappa}$ and  $F, G\in \fin$, with $|F|, |G| > n$, such that for every $x \in \ada'$, $F \subseteq x$ and $G \subseteq h(x)$. If $\ada'' \in [\ada']^\kappa$, apply Lemma \ref{L2} to $h\rest_{\ada''} : \ada'' \to h[\ada'']$ and get $x,y,z \in \ada''$ such that $x\cap y \subseteq x \cap z$ and $h(x)\cap h(y) \subseteq h(x) \cap h(z)$. Thus, there are $a_0,a_1 \in A\setminus n$ and $ b_0, b_1 \in B \setminus n$ such that $|x\cap z \setminus x\cap y| = a_0 -a_1  \neq b_0 -b_1 = |h(x)\cap h(z)\setminus h(x)\cap h(y)|.$
\end{proof}

\begin{coro}\label{C2}
Let $\ada, \adb$ be AD families and $h: \ada \to \adb$ a bijection. If $\{ |x\cap y| : x,y \in \ada \}$ and $\{ |x\cap y| : x,y \in \adb \}$ are almost oscillating, there is no homeomorphism from $\Psi (\ada)$ to $\Psi (\adb)$ that extends $h$.
\end{coro}

\begin{proof}
If $H:\Psi (\ada) \to \Psi (\adb)$ is such homeomorphism, by Corollary \ref{C1} there is $\ada' \in [\ada]^{\kappa}$ such that $H\rest_{\ada'} : \ada' \to H[\ada']$ is of dense oscillation. If $W = \bigcup_{A \in \ada'} A$, then $Z= \ada' \cup W$ is a subspace of $\Psi (\ada)$ and $H\rest_Z$ is a homeomorphism contradicting Proposition \ref{P1}.
\end{proof}

\section{$\mathfrak{c}$ many types of $\Psi$-spaces}

Next we construct $\mathfrak{c}$ many AD families whose $\Psi$-spaces are pairwise non-homeo\-mor\-phic for each of the classes of Luzin families and branch families of $\cantor$. 

\begin{theorem}\label{T1}
There are $\mathfrak{c}$ different Luzin families with non-homeomorphic $\Psi$-spaces. 
\end{theorem}

\begin{proof}
Given $L = \{ k_n : n\in \omega  \} \subseteq \omega$ such that $k_n > \sum_{i < n} k_i$, construct a Luzin family $\ada_L$ as follow: Choose a partition $\{A_n : n\in \omega \}$ of $\omega$ into infinite sets. Suppose constructed $A_\beta$, $\beta < \alpha$. Enumerate $\{A_\beta : \beta < \alpha \}$ as $\{ B_n : n\in \omega \}$  and for each $n \in \omega$, pick $a_n \subseteq B_n \setminus \bigcup_{i < n} B_i$ such that $|(\bigcup_{i \leq n} a_i ) \cap B_n| = k_n$. Let $A_\alpha = \bigcup_{n\in \omega} a_n$ and $\ada_L=\{A_\alpha:\alpha<\omega_1\}$. Observe that for any $\omega < \alpha, \beta < \omega_1$, there is $n\in \omega$ with  $|A_\alpha \cap A_\beta| = k_n$. It is easy to see that $\ada_A = \{ A_\alpha : \omega < \alpha < \omega_1 \}$ is a Luzin family. Construct all the Luzin families below from the same partition $\{A_n : n\in \omega \}$.

Now, by Proposition \ref{P2}, let $\{ P_\alpha : \alpha < \mathfrak{c} \}$ be a pairwise almost oscillating family of sets of $\omega$. For every $\alpha < \mathfrak{c}$, let $Q_\alpha = \{ q_n^\alpha : n\in \omega \} \subseteq P_\alpha$ such that for every $n\in \omega$, $q_n ^\alpha > \sum_{i<n} q_i^\alpha$. Notice that $\{Q_\alpha : \alpha < \mathfrak{c} \}$ is also a pairwise almost oscillating family of sets of $\omega$. By Corollary \ref{C2}, $\{ \ada_{Q_\alpha} : \alpha < \mathfrak{c}\}$ is the desire collection of Luzin families.
\end{proof}

\begin{theorem}\label{T2}
Given a cardinal $\kappa \leq \mathfrak{c}$ of uncountable cofinality, there are $\mathfrak{c}$ different homeomorphic subsets of $\cantor$ of size $\kappa$ with non-homeomorphic $\Psi$-spaces.
\end{theorem}
\begin{proof}
Given $A \in [\omega]^\omega$, consider the tree $S_A \subseteq \dosfin$ defined by $\emptyset \in S_A$ and
$$  s\in Lev_n (S_A) \implies (s^\frown 1 \in S_A) \wedge (s^\frown 0 \in S_A \longleftrightarrow n \in A).$$

Let $X$ be any subset of size $\kappa$ of the set of branches $[S_A] \subseteq \cantor$. Notice that for all $x,y\in X$, $\Delta (x,y) = |\gorro{x}\cap \gorro{y}| \in A$. 

Again, by Proposition \ref{P2}, let $\{ P_\alpha : \alpha < \mathfrak{c} \}$ be a pairwise almost oscillating family of sets of $\omega$. Note that if $A, B  \in [\omega]^\omega$, then $[S_A] \simeq [S_B] \simeq\cantor$ and $A \cap B =^* \emptyset$ implies that $|[S_A] \cap [S_B]| < \omega$. Hence, we can choose $X_\alpha \in [[S_{P_\alpha}]]^{\kappa}$ such that the $X_\alpha$'s are all different, but $X_\alpha \simeq X_\beta$, whenever $\alpha, \beta < \mathfrak{c}$. By Corollary \ref{C2}, $\{ X_{\alpha} : \alpha < \mathfrak{c}\}$ is the desire collection of subsets of $\cantor$. 
\end{proof}


\begin{coro}\label{C3}
Let $\ada$ be an AD family of size $\kappa$. If there are $\ada_0, \ada_1 \in [\ada]^{\kappa}$ such that $\{ |x\cap y| : x,y \in \ada_0 \}$ and $\{ |x\cap y| : x,y \in \ada_1 \}$ are almost oscillating, then $\Psi (\ada) \not\simeq \Psi (\ada_0)$.
\end{coro}

\begin{proof}
If $h : \ada_0 \to \ada$ is a bijection, by Corollary \ref{P1} there is $\ada_0' \in [h^{-1}[\ada_1]]^{\kappa}$ such that $h\rest_{\ada_0'} : \ada_0' \to h[\ada_0']$ is of dense oscillation. Hence, there can not be a homeomorphism between $\Psi (\ada_0 ')$ and $\Psi (h[\ada_0'])$ that extends $h\rest_{\ada_0'}$, then neither one between $\Psi (\ada_0)$ and $\Psi (\ada)$ that extends $h$.
\end{proof}

Motivated by Corollary \ref{C3}, a positive answer to the following question implies a positive answer to Question \ref{Q2}. However, we do not even know if $\mathsf{CH}$ answers:

\begin{pregunta}\label{Q3}
Let $\ada$ be an AD family on $\omega$ of size $\omega_1$. Are there $\ada_0, \ada_1 \in [\ada]^{\omega_1}$ such that $\{ |x\cap y| : x,y \in \ada_0 \}$ and $\{ |x\cap y| : x,y \in \ada_1 \}$ are almost oscillating?
\end{pregunta}

The arguments under $\mathsf{CH}$ below Question \ref{Q2} say that if $\ada$ is an uncountable AD family, then there is $\ada_0$ such that $\Psi (\ada) \not\simeq \Psi (\ada_0)$. However, the sets $\{ |x\cap y| : x,y \in \ada_0 \}$ and $\{ |x\cap y| : x,y \in \ada \}$ are far from being almost oscillating (the first is contained in the second).


\end{document}